\documentclass[a4paper]{article}
\usepackage{amsthm,amsfonts,amsmath,amssymb}
\usepackage[abbrev,nobysame]{amsrefs}
\usepackage[cp1251]{inputenc}
\usepackage[english]{babel}
\usepackage[final]{graphicx}
\usepackage{setspace}
\usepackage[12pt]{extsizes}
\begin{document}
\newtheorem{teorema}{Theorem}
\newtheorem{lemma}{Lemma}
\newtheorem{utv}{Proposition}
\newtheorem{svoistvo}{Property}
\newtheorem{sled}{Corollary}
\newtheorem{con}{Conjecture}
\newtheorem{zam}{Remark}
\newtheorem{quest}{Question}

\author{A. A. Taranenko\thanks{Sobolev Institute of Mathematics, Novosibirsk, Russia; \texttt{taa@math.nsc.ru}}}
\title{On a metric property of perfect colorings}
\date{February 3, 2021}

\maketitle

\begin{abstract} 
Given a perfect coloring of a graph, we prove that  the $L_1$ distance between two rows of the adjacency matrix of the graph is not less than the $L_1$ distance between the corresponding rows of the parameter matrix of the coloring. With the help of an algebraic approach, we deduce corollaries of this result for perfect $2$-colorings, perfect colorings in distance-$l$ graphs and in distance-regular graphs. We also provide examples when the obtained property reject several putative parameter matrices of perfect colorings in infinite graphs.

\end{abstract}

\section{Definitions and main result}

Following~\cite{my.perfstrct}, we consider perfect colorings in a more general setting than perfect colorings of simple graphs or multigraphs.
More specifically, we associate a \textit{graph} $G$ on $n$ vertices with a real $n \times n$-matrix $M$ that is called its \textit{adjacency matrix}. So under a graph we mean an oriented graph with edges labelled by $m_{u,v}$. We use $V(G)$ to denote the vertex set of the graph $G$.

Define a \textit{perfect $k$-coloring}  of a graph $G$ with the \textit{parameter matrix} $S$ to be a partition of the set $V(G)$ into disjoint classes $J_i$, $i = 1, \ldots, k$, such that for all $u \in J_i$ it holds $s_{i,j} = \sum\limits_{v \in J_j} m_{u,v} $. This definition of a perfect coloring generelizes  equitable partitions  introduced by Delsarte~\cite{dels.assch}.

Note that classes $J_i$ can be considered as vertices of some graph $H$ defined by the adjacency matrix $S$. Then  a perfect coloring  can be defined as a map $f: V(G) \rightarrow V(H)$ that puts every vertex $v$ of the graph $G$ to its color $f(v)$, i.e., $f(v) = i$ if and only if $v \in J_i$. 

In many papers, a perfect coloring of simple graph $G$  is said to be  a partition of its vertex set  into color classes such that the colored neighborhood of each vertex is defined by a color of the vertex.

At last, perfect colorings can be treated as special perfect structures~\cite{my.perfstrct}: a perfect coloring is a triple of matrices $(M,P,S)$ connected by a relation $MP = PS$, where $M$ and $S$ are square matrices of orders $n$ and $k$ respectively,  $P$ is a $(0,1)$-matrix of sizes $n \times k$ in which each row contains exactly one unity entry.

In the present note, we bound the $L_1$ distance between rows of the parameter matrix of a perfect coloring by the $L_1$ distance between the corresponding rows of the adjacency matrix of a graph.
 
Recall that the \textit{$L_1$ distance} between two $n$-tuples $x = (x_1, \ldots, x_n)$ and $y = (y_1, \ldots, y_n)$, $x_i, y_i \in \mathbb{R}$,  is
$$d(x,y) = \sum\limits_{i=1}^n |x_i - y_i|.$$
Given a real matrix $A$ of order $n$, we use $[A]^i$ to denote the $i$-th row of $A$.

\begin{teorema} \label{nonexpandth}
Let $f$ be a perfect $k$-coloring  of a graph $G$ with the parameter matrix $S$  and let $M$ be the adjacency matrix of $G$. Then for all $u, v \in V(G)$ we have
$$d([M]^u, [M]^v) \geq d([S]^{f(u)}, [S]^{f(v)}).$$ 
\end{teorema}

\begin{proof}
By the definition,
$$d([S]^{f(u)}, [S]^{f(v)}) = \sum\limits_{j=1}^k |s_{f(u),j} - s_{f(v),j}|.$$
Since $f$ is a perfect coloring, the set $V(G)$  is partitioned into $k$ disjoint subsets $J_1, \ldots, J_k$ such that for every $u \in  V(G)$ we have $s_{f(u),j} = \sum\limits_{w \in J_j} m_{u,w}$. Consequently,
$$ \sum\limits_{j=1}^k |s_{f(u),j} - s_{f(v),j}| =   \sum\limits_{j=1}^k |\sum\limits_{w \in J_j}(m_{u,w} - m_{v,w})|.$$
Using the inequality $|a + b| \leq |a| + |b|$, we deduce 
$$  \sum\limits_{j=1}^k |\sum\limits_{w \in J_j}(m_{u,w} - m_{v,w})| \leq \sum\limits_{w \in V(G)} |m_{u,w} - m_{v,w}|.$$
It only remains to note that 
$$\sum\limits_{w \in V(G)} |m_{u,w} - m_{v,w}| = d([M]^u, [M]^v).$$
\end{proof}

\section{Corollaries for simple graphs}

Let us specialize Theorem~\ref{nonexpandth} for some classes of graphs and colorings. In this section, we assume everywhere that $G$ is an $r$-regular  simple undirected graph with no loops. In other words, the adjacency matrix $M$ of $G$ is a symmetric $(0,1)$-matrix with zeroes within the main diagonal and row sums equal to $r$. 

Given a vertex $v$ of a simple graph $G$, let $\mathcal{N}(v)$ denote the \textit{neighborhood} of the vertex $v$ that is the set of all vertices $u \in V(G)$ such that $u$ and $v$ are adjacent. Then for simple graphs $G$ Theorem~\ref{nonexpandth}  takes  the following form.

\begin{teorema} \label{thingraphs}
Let $G$ be a simple $r$-regular graph and $f$ be a perfect coloring of $G$ with the parameter matrix $S$. Assume that there are vertices $u, v \in V(G)$ of colors $f(u) = i$, $f(v) = j$  such that $|\mathcal{N}(u) \cap \mathcal{N}(v)| = h$. Then
$$d([S]^i, [S]^j) \leq 2(r-h).$$
If the inequality becomes an equality, then  distributions of colors in sets $\mathcal{N}(u) \cap \mathcal{N}(v)$, $\mathcal{N}(u) \setminus \mathcal{N}(v)$, and $\mathcal{N}(v) \setminus \mathcal{N}(u)$ are determined by the parameter matrix $S$ and do not depend on the coloring $f$.
\end{teorema}

\begin{proof}
The theorem follows from Theorem~\ref{nonexpandth} and the fact that $M(u,v) = 2(r - |\mathcal{N}(u) \cap \mathcal{N}(v)|)$. 

For perfect colorings of simple graphs the equality $d([M]^u, [M]^v) = d([S]^{f(u)}, [S]^{f(v)})$  means that the symmetric difference $\mathcal{N}(u) \Delta \mathcal{N} (v)$ contains exactly $|s_{f(u),j} - s_{f(v), j}|$ vertices of each color $j$. Since $[S]^{f(u)}$ and $[S]^{f(v)}$ are  distributions of colors  in sets  $\mathcal{N}(u)$ and $\mathcal{N}(v)$ respectively,  we know color distributions for sets $\mathcal{N}(u) \cap \mathcal{N}(v)$, $\mathcal{N}(u) \setminus \mathcal{N}(v)$, and $\mathcal{N}(v) \setminus \mathcal{N}(u)$.
\end{proof}

\subsection{Perfect $2$-colorings}

The parameter matrix of perfect colorings in $2$ colors is usually written as
$$S = \left( \begin{array}{cc} a & b \\ c & d \end{array} \right).$$
It is easy to see that if $G$ is an $r$-regular graph, then $S$ has two different eigenvalues: the trivial eigenvalue $\lambda_1 = r$ and the \textit{second eigenvalue} $\lambda_2 = r - (b+c) = a - c$. Since $a + b = c +d = r$, the parameters $b$ and $c$ uniquely define the matrix $S$. Thus we will say that a perfect coloring in colors with the parameter matrix is a \textit{$(b,c)$}-coloring.

\begin{lemma}  \label{2colordist}
Let 
$$S = \left( \begin{array}{cc} a & b \\ c & d \end{array} \right)$$
be the parameter matrix of a perfect $(b,c)$-coloring of an $r$-regular graph $G$.
Then 
$$d([S]^1, [S]^2) = 2 |\lambda_2| = 2 |r - (b+c)|.$$
\end{lemma} 
\begin{proof}
By  equalities $a + b = c +d = r$ and $\lambda_2 = r - (b+c)$, we have
$$d([S]^1, [S]^2) =  |a - c| + |b - d| = 2 |r - (b+c)| = 2 |\lambda_2|.$$
\end{proof}

For further applications, we state Theorem~\ref{thingraphs} for perfect $(b,c)$-colorings.

\begin{teorema} \label{2coloringr}
Let $G$ be an $r$-regular graph and $f$ be a perfect $(b,c)$-coloring of $G$. Assume that there are vertices $u, v \in V(G)$ of different colors such that $|\mathcal{N}(u) \cap \mathcal{N}(v)| = h$. Then
$$ h \leq b+c \leq 2r - h. $$
If the left inequality attains an equality, then all vertices from $\mathcal{N}(u) \setminus\mathcal{N} (v)$ (and $\mathcal{N}(v) \setminus\mathcal{N} (u)$) have  the same color as the vertex $u$ (vertex $v$).  If the right inequality is achieved, then all vertices from $\mathcal{N}(u) \setminus\mathcal{N} (v)$ (or $\mathcal{N}(v) \setminus\mathcal{N} (u)$) have  the same color as the vertex $v$ (vertex $u$).  

At last, if the vertices $u$ and $v$ are adjacent, then $h+2 \leq b+c$.
\end{teorema}

\begin{proof}
By Theorem~\ref{thingraphs}, we have
$$d([S]^{f(u)},[S]^{f(v)} ) \leq 2(r-h).$$
Since vertices $u$ and $v$ have different colors, Lemma~\ref{2colordist} gives
$$d([S]^{f(u)},[S]^{f(v)} ) =d([S]^{1},[S]^{2} ) = 2|r - (b+c)| \leq 2(r-h),$$
that is equivalent to the required inequalities. 

Equality  $2(r - (b+c)) =  2(r-h)$ means that the sum $b+c$ attains the minimal possible value. Then the set $\mathcal{N}(u)$ (and $\mathcal{N}(v)$) contains the minimal possible number of vertices whose color is different from $f(u)$ ($f(v)$). Similarly, if  $2(r - (b+c)) =  -2(r-h)$, then the sum $b+c$ attains the maximal possible  value, and $\mathcal{N}(u)$ ($\mathcal{N}(v)$) contains the maximal possible number of vertices with colors different from $f(u)$ ($f(v)$). Thus sets $\mathcal{N}(u) \setminus\mathcal{N} (v)$ and $\mathcal{N}(v) \setminus\mathcal{N} (u)$  are monochromatic. 

If vertices $u$ and $v$ are adjacent, then we can slightly improve the inequalities. Indeed, in this case  the sets $\mathcal{N}(u) \setminus\mathcal{N} (v)$ and $\mathcal{N}(v) \setminus\mathcal{N} (u)$  contain at least one vertex of color $\mathcal{N}(v)$ and $\mathcal{N}(u)$, respectively. It means that $2(r - (b+c)) \leq  2(r-h -2)$.
\end{proof}

\subsection{Colorings in distance graphs}

As before, we use $M$ for the adjacency matrix of a simple graph $G$. It is well known, that  powers of $M$ count the number of paths from one vertex of $G$ to another: the number of paths of length $l$ in $G$ from a vertex $u$ to a vertex $v$ is equal to the $(u,v)$-entry of the matrix $M^l$. We will say that $M^l$ is the adjacency matrix of the \textit{distance-$l$ graph $G^l$}. Note that in most cases $G^l$ is not a simple graph but a multigraph. 

In~\cite{my.perfstrct} it was proved the following.

\begin{utv} \label{perfinpoly}
If a triple of matrices $(M,P,S)$ is a perfect coloring, then for every polynomial $p (x)\in \mathbb{R}[x]$ the triple $(p(M), P, p(S))$ is also a perfect coloring. 
\end{utv}

So we can specialize Theorem~\ref{nonexpandth}  for distance-$l$ graphs. 

\begin{teorema} \label{distgraphs}
Let $G$ be a simple graph and $f$ be a perfect coloring of $G$ with the parameter matrix $S$. Then for all $l \in \mathbb{N}$ and for all vertices $u,v \in V(G)$ it holds
$$ d([M^l]^u, [M^l]^v) \geq   d([S^l]^{f(u)}, [S^l]^{f(v)}) .$$
\end{teorema}

\begin{proof}
The result follows from Theorem~\ref{nonexpandth} and Proposition~\ref{perfinpoly}.
\end{proof}

In special graphs, it is possible to express some subsets of vertices by the means of a polynomial on the adjacent matrix. One of the most famous examples of such sets and graphs are balls and spheres in distance-regular graphs.

For vertices $u,v$ of a simple graphs $G$, let $\rho(u,v)$ denote the distance between $u$ and $v$ (the length of the shortest path between them). The \textit{ball} $B_r(u)$ of a radius $r$ and with center  $u$ is a set $\{ v : \rho(u,v) \leq r \}$, and the \textit{sphere} $W_r(u)$ of a radius $r$ and with center  $u$ is a set $\{ v  :\rho(u,v)  =  r \}$. 

It is well known (see, e.g.~\cite{BrCohNeu.distreggraph}) that in a distance-regular graph $G$ for every $r \leq diam(G)$ there are polynomials $p_r^B$ and $p_r^W$ such that for each $u \in V(G)$  rows $[p^B_r (M)]^u$ and $[p^W_r (M)]^u$ are the indicator functions of a ball $B_r(u)$ and a sphere $W_r(u)$, respectively.   Thus, for distance-regular graphs we have the following theorem.

\begin{teorema} \label{distreggraphs}
Let $G$ be a simple distance-regular graph with polynomials $p_r^B(x)$ and $p_r^W(x)$ expressing balls and spheres of radius $r$ in $G$, respectively.  Suppose that $f$ is a perfect coloring of $G$ with the parameter matrix $S$. Then  for all vertices $u,v \in V(G)$ 
\begin{gather*}
|B_r(u) \Delta B_r(v)|  \geq   d([p_r^B(S)]^{f(u)}, [p_r^B(S)]^{f(v)}); \\
|W_r(u) \Delta W_r(v)|  \geq   d([p_r^W(S)]^{f(u)}, [p_r^W(S)]^{f(v)}).
\end{gather*}
\end{teorema}

\section{Applications and examples} \label{exampsec}

The above results can be applied to reject some putative parameter matrices of perfect colorings for a given graph $G$. For perfect $2$-colorings, our method is more useful if the second eigenvalue of the parameter matrix has a large absolute value. It is especially interesting for infinite graphs (i.e., graphs with an infinite number of vertices) because the standard spectral condition on the existence of perfect colorings is not applicable for them.

\subsection{Square and triangular grids}

Our first example is a simple proof that there are no perfect $(4,3)$-colorings of the square grid. The \textit{square grid} is an infinite $4$-regular graph with the vertex set $\mathbb{Z}^2$ and edges $((x,y), (x+1, y))$ and $((x,y), (x, y+1))$ for $x, y \in \mathbb{Z}$. 

Suppose that $f$ is a perfect $(4,3)$-color of the square grid. By Theorem~\ref{2coloringr}, there are no vertices $u = (x,y)$ and $v = (x+1, y+1)$ in the square grid of different colors in the coloring $f$ because $h = |\mathcal{N}(u) \cap \mathcal{N}(v) | = 2$ and  $7 = b + c > 2r - h = 6$. So for a given vertex $(x,y)$ all vertices $(x + t, y +t)$, $t \in \mathbb{Z}$ have the same color in $f$ as the vertex $(x,y)$. Then every vertex is adjacent to an even number of vertices of each color that contradicts to the parameters of  $f$.

An approach similar to the presented one was used in a characterization of $3$-colorings in the square grid~\cite{puzynina.3perfsqgr} and in studying multiple coverings of the square grid with balls of a constant radius~\cite{axenovich.multcovsqgr}.

We can also apply this technique to perfect $2$-colorings in the triangular grid that is an infinite $6$-regular graph with the following local structure:
\begin{center}
		\includegraphics[width=0.3\linewidth]{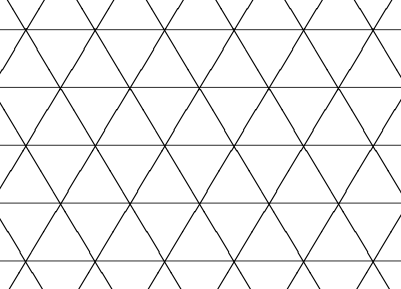}
\end{center}

For every adjacent vertices $u$ and $v$ in the triangular grid we have $h = |\mathcal{N}(u) \cap \mathcal{N}(v) | = 2$. By Theorem~\ref{2coloringr}, for every perfect $(b,c)$-coloring of the triangular  it holds $4 \leq b + c \leq 10$. In particular, there are no perfect $(1,1)$-, $(2,1)$-, $(6,5)$-, and $(6,6)$-colorings in the triangular grid.

One can also show that there are no perfect $(3,1)$-, $(5,5)$-, and $(6,4)$-colorings of the triangular grid (when $b+c$ achieves one of the possible equalities). Indeed, Theorem~\ref{2coloringr} allows us to determine colors of vertices in sets $\mathcal{N}(u) \setminus (\mathcal{N}(v) \cup \{ v\})$ and $\mathcal{N}(v) \setminus (\mathcal{N}(u) \cup \{ u\})$ that gives a contradiction to the parameters of the coloring at the vertex $w$:
\begin{center}
	\includegraphics[width=0.3\linewidth]{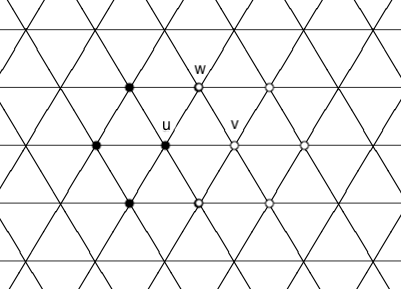} ~~~ 	\includegraphics[width=0.3\linewidth]{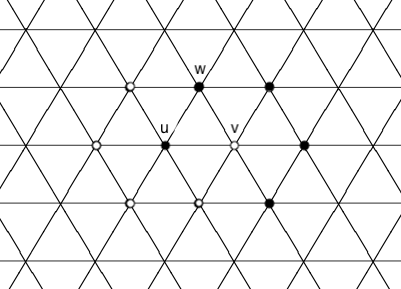}  ~~~
		\includegraphics[width=0.3\linewidth]{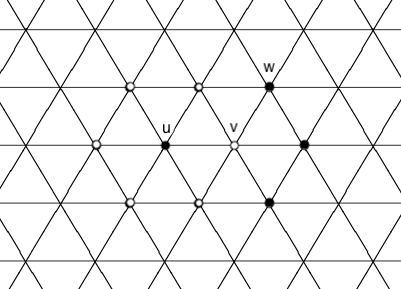} 
\end{center}

The similar reasoning implies that there is a unique (up to transformations of the plane) perfect $(2,2)$-coloring of the triangular grid:

\begin{center}
		\includegraphics[width=0.3\linewidth]{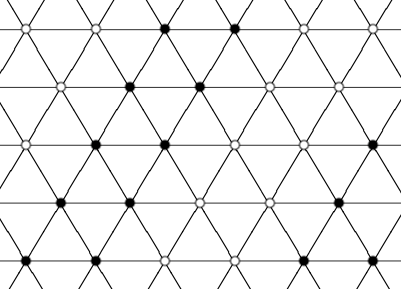}
\end{center}

For more information on perfect $2$-colorings of the triangular grid see~\cite{puzynina.perhextri}.

\subsection{Circulant graphs}

Our results can be widely used for perfect colorings in circulant graphs.  Perfect colorings in some classes of such graphs were previously studied in~\cite{parshina.circcontsist,ParLis.circcontoddsist}.

Given a (multi)set  $D = \{d_1, \ldots, d_m\}$, $d_i \in \mathbb{N}$, a \textit{circulant} (multi)graph $C(d_1, \ldots, d_m)$ is a $2m$-regular (multi)graph with the vertex set $\mathbb{Z}$ and edges $(x,y)$, where $|x - y| \in D$. It is easy to see that every perfect coloring $f$ of $C(d_1, \ldots, d_m)$ is periodic: there is some $T \in \mathbb{N}$ such that $f(x + T ) = f(x)$ for all $x \in \mathbb{Z}$. 
 
\begin{teorema}
Assume that for a (multi)set $D = \{ d_1, \ldots, d_m \}$ and for  $t \in \mathbb{N}$ we have 
$$|\{ \pm d_1, \ldots, \pm d_m\} \cap \{ t \pm d_1, \ldots, t \pm d_m\} |= h .$$
If $b +c > 4m - h$ or $b +c < h$ (or $b+c < h+2$ if $t  \in D$), then the period $T$ of every perfect  $(b,c)$-coloring of the circulant (multi)graph $C(d_1, \ldots, d_m)$ divides $t$.
\end{teorema}

\begin{proof}
Let $f$ be a perfect $(b,c)$-coloring of the (multi)graph $C(d_1, \ldots, d_m)$  with period $T$.  If $T$ is not a divisor of $t$, then there are vertices $x$ and $x + t$ in $C(d_1, \ldots, d_m)$  that are colored with different colors by $f$. By the condition of the theorem, $|\mathcal{N}(x) \cap \mathcal{N}(x +t) | = h$.  It only remains to note that the demanded inequalities on $b+c$ contradict to Theorem~\ref{2coloringr}.
\end{proof}

For example, consider a circulant graph $C(1,2,4)$. For $t = 3$ we have 
$$ h = |\{ \pm 1,  \pm 2, \pm 4 \} \cap \{ 3 \pm 1, 3 \pm 2, 3 \pm 4 \}| = 4. $$
So all $(b,c)$-colorings of $C(1,2,4)$ with  $b+c < 4$ or $b+c >8$ have a period $T$ such that $T$ divides $3$, and so $T = 3$. Searching all colorings of $C(1,2,4)$ of period $3$, it is easy to see that there are no $(1,1)$-, $(2,1)$-, $(5,4)$-, $(5,5)$-, $(6,4)$-, $(6,5)$-, and $(6,6)$-colorings among them. Therefore, these colorings do not exist in $C(1,2,4)$.

\section*{Acknowledgements}

This work was carried out within the framework of the state contract
of the Sobolev Institute of Mathematics (project no. 0314-2019-0016).

\end{document}